\documentclass[11pt,a4paper]{article}

\usepackage[margin=2.2cm]{geometry}

\usepackage[latin1]{inputenc}
\usepackage{amsmath,amssymb,amsthm}

\usepackage{bbm}
\usepackage{enumitem}

\usepackage[colorlinks=true,citecolor=black,linkcolor=black,urlcolor=blue]{hyperref}
\usepackage{graphicx}
\usepackage{overpic}

\usepackage{chngcntr}

\usepackage[numbers]{natbib}
\usepackage[font=small,labelfont=bf]{caption}

\newtheorem{theorem}{Theorem}
\newtheorem*{theorem*}{Theorem}
\newtheorem*{lemma*}{Lemma}
\newtheorem{lemma}[theorem]{Lemma}

\newtheorem*{observation*}{Observation}
\newtheorem*{claim*}{Claim}
\newtheorem{claim}{Claim}

\renewcommand{\le}{\leqslant}
\renewcommand{\ge}{\geqslant}
\renewcommand{\epsilon}{\varepsilon}

\newcommand{\E}{\mathbb{E}}
\newcommand{\Pb}{\mathbb{P}}

\newcommand{\Z}{\mathbb{Z}}
\newcommand{\N}{\mathbb{N}}
\newcommand{\R}{\mathbb{R}}
\newcommand{\Bin}{\mathrm{Bin}}

\newcommand{\mc}{\mathcal}

\newcommand{\Poi}{{\mathrm{Poi}}}
\newcommand{\dtv}{{d_{\mathrm{TV}}}}
\newcommand{\imax}{{i_{\mathrm{max}}}}

 \title{Non-concentration of the chromatic number of a random graph 
 }

\author{Annika Heckel
\thanks{Mathematisches Institut der Universit\"at M\"unchen, Theresienstr.\ 39, 80333 M\"unchen, Germany. E-mail: 
\texttt{heckel@math.lmu.de}. This research was funded by ERC Grants 676632-RanDM and 772606-PTRCSP.
}
}
\date{\today 
}
\begin{document}
 \maketitle 
\begin{abstract}
We show that the chromatic number of $G_{n, \frac 12}$ is not concentrated on fewer than $n^{\frac 14-\epsilon}$ consecutive values. This addresses a longstanding question raised by Erd\H os and several other authors.
\end{abstract}

\section{Introduction}
Random graph theory was initiated in the late 1950s and early 1960s in the pioneering works of Erd\H os and R\'enyi~\cite{erdos1959random, erdos1960evolution}. For the \emph{binomial random graph} $G_{n,p}$, include each possible edge between $n$ labelled vertices independently with probability $p$. For the closely related \emph{uniform random graph} $G_{n,m}$, choose a set of exactly $m$ edges uniformly at random from all edge sets of size $m$. Both of these models have been studied extensively since their introduction sixty years ago, and we refer to the standard texts \cite{bollobas:randomgraphs} and \cite{janson:randomgraphs} for the rich history of this subject and many impressive results.

The \emph{chromatic number} of a graph $G$, denoted by $\chi(G)$,
is a central concept both in graph theory in general and in random graphs in particular. It is defined as the minimum number of colours needed for a vertex colouring of $G$ where no two adjacent vertices are coloured the same. The study of the chromatic number of random graphs goes back to the foundational papers by Erd\H os and R\'enyi \cite{erdos1960evolution} and includes some of the most 
celebrated results in random graph theory. 

The case of dense random graphs where $p=\frac 12$ has received particular attention. Grimmett and McDiarmid \cite{grimmett1975colouring} first found the order of magnitude of $\chi(G_{n, \frac 12})$ in 1975,
and in a breakthrough paper in 1987, Bollob\'as \cite{bollobas1988chromatic} used martingale concentration arguments to 
establish the asymptotic value.
\begin{theorem}[\cite{bollobas1988chromatic}]\label{theorem:bollobas}With high probability\footnote{As usual, we say that a sequence $(E_n)_{n \in \N}$ of events holds \emph{with high probability (whp)} if $\Pb(E_n) \rightarrow 1$ as $n \rightarrow \infty$.},
\[\chi(G_{n, \frac 12}) \sim \frac{n}{2 \log_2 n}.\] 
\end{theorem}
Several improvements to these bounds were made by McDiarmid \cite{mcdiarmid1989method}, Panagiotou and Steger~\cite{panagiotou2009note} and Fountoulakis, Kang and McDiarmid \cite{fountoulakis2010t}. In 2016 \cite{heckel2018chromatic}, the author used the second moment method, combined with martingale concentration arguments, to obtain the following result.
\begin{theorem}[\cite{heckel2018chromatic}]\label{theorem:bounds}
With high probability,  
\[
 \chi(G_{n, \frac 12}) = \frac{n}{2 \log_2 n-2 \log_2 \log_2 n - 2}+o\left(\frac{n}{\log^2 n}\right).
\]
\end{theorem}

While these bounds give an explicit interval of length $o \left( \frac n {\log^2 n} \right)$ which contains $\chi(G_{n, \frac 12})$ whp, much narrower concentration is known to hold. A remarkable result of Shamir and Spencer~\cite{shamir1987sharp} states that for \emph{any} sequence $p=p(n)$, $\chi(G_{n,p})$ is whp contained in a (non-explicit) sequence of intervals of length about $\sqrt{n}$. For $p = \frac 12$, Alon improved this slightly to about $\frac{\sqrt{n}}{\log n}$ (this is \S 7.9, Exercise 3 in \cite{alonspencer}, see also \cite{scott2008concentration}).

For sparse random graphs, much more is known: Shamir and Spencer~\cite{shamir1987sharp} also showed that for $p<n^{-\frac 56 -\varepsilon}$, $\chi(G_{n,p})$ is whp concentrated on only {five} consecutive values; {\L}uczak \cite{luczak1991note} improved this to two consecutive values and finally Alon and Krivelevich~\cite{alon1997concentration} showed that two point concentration holds for $p<n^{-\frac 12-\varepsilon}$. In a landmark contribution, Achlioptas and Naor \cite{achlioptas2005two} found two \emph{explicit} such values for $p=d/n$ where $d$ is constant, and Coja-Oghlan, Panagiotou and Steger \cite{coja2008chromatic} extended this to three explicit values for $p<n^{-\frac34 - \epsilon}$.

However, while there is a wealth of results asserting sharp concentration of the chromatic number of $G_{n,p}$, until now there have been no non-trivial cases where $\chi(G_{n,p})$ is known \emph{not} to be extremely narrowly concentrated. (Though Alon and Krivelevich \cite{alon1997concentration} note that it is trivial that $\chi(G_{n,p})$ is not concentrated on fewer than $\Theta(\sqrt{n})$ values for $p=1-1/(10n)$.) 

In his appendix to the standard text on the probabilistic method~\cite{alonspencerfirstedition}, Erd\H os asked the following question (see also \cite{chung1998erdos}): How accurately can $\chi(G_{n, \frac 12})$ be estimated? 
Can it be shown \emph{not} to be concentrated on a series of intervals of constant length? Variants of this question are discussed in \cite{alon1997concentration}, \cite{bollobas:randomgraphs}, \cite{glebov2015concentration} and \cite{mcdiarmidsurvey}. 
In 2004,  Bollob\'as \cite{bollobas:concentrationfixed} asked for any non-trivial examples of non-concentration of the chromatic number of random graphs, specifically suggesting the dense random graph $G_{n,m}$ with $m= \left \lfloor n^2/4 \right \rfloor$ (which corresponds to $p=\frac 12$) as a candidate.  
He mentions discussing the question frequently with Erd\H{o}s in the late 1980s, and notes that ``even the weakest results claiming lack of concentration would be of interest.''

In this paper, we show that $\chi(G_{n, \frac 12})$ is not whp concentrated on fewer than $n^{\frac 14 - \varepsilon}$ consecutive values. As a corollary, the same is true for the random graph $G_{n,m}$ with $m= \left \lfloor n^2/4 \right \rfloor$; more details are given in Section \ref{section:remarks}.
\begin{theorem}\label{theorem:nonconcentration}
 For any constant $c< \frac 14$, there is no sequence of intervals of length $n^{c}$ which contain $\chi(G_{n, \frac 12})$ with high probability. 
\end{theorem}

 The proof of Theorem \ref{theorem:nonconcentration} is based on the close relationship between the chromatic number and the number of maximum independent sets in $G_{n, \frac 12}$. 
 For a graph $G$, we denote by $\alpha(G)$ the \emph{independence number} of $G$, that is, the size of the largest independent vertex set. The independence number of $G_{n, \frac 12}$ is very well studied: let
\begin{equation} \label{eq:adef}
 \alpha_0 = \alpha_0(n)= 2 \log_2 n - 2 \log_2 \log_2 n +2 \log_2 \left( e/2 \right)+1 \,\,\text{ and }\,\, a=a(n)=\left \lfloor \alpha_0 \right \rfloor,
\end{equation}
then it follows from the work of Matula \cite{matula1970complete,matula1972employee} and Bollob\'as and Erd\H os \cite{erdoscliques} that whp $\alpha(G_{n, \frac 12}) = \left\lfloor \alpha_0+o(1)\right \rfloor$, pinning down $\alpha(G_{n, \frac 12})$ to at most two consecutive values. In fact, for most $n$, whp $\alpha(G_{n, \frac 12}) = a$. 

In the following, we will call a set of vertices of size $a$ an \emph{$a$-set}. Let $X_a$ denote the number of independent $a$-sets in $G_{n, \frac 12}$, then the distribution of $X_a$ is known to be approximately Poisson with mean $\mu=\mathbb{E}[X_a]$ (for details see Section \ref{section:preliminaries}). We will see in Section \ref{section:preliminaries} that
\[
 \mu = {n \choose a} \left(\frac 12\right)^{{a \choose 2}}=n^{x} \,\,\text { where }\,\, o(1) \leqslant x(n)\leqslant 1+o(1).
\]
 In particular, $X_{a}$ is not whp contained in any sequence of intervals of length shorter than $\sqrt{\mu}=n^{x/2}$.

Note that $\chi(G_{n, \frac 12})$ is closely linked to $\alpha(G_{n, \frac 12})$. The lower bound in Theorem \ref{theorem:bollobas} comes from the simple relationship $\chi(G) \ge n / \alpha(G)$ which holds for any graph $G$ on $n$ vertices, and Theorem \ref{theorem:bounds} implies that the average colour class size in an optimal colouring of $G_{n, \frac 12}$ is 
\[\alpha_0-1-\frac 2 {\log 2}+o(1) \approx \alpha_0 - 3.89.\]
It is plausible that an optimal colouring of $G_{n, \frac 12}$ contains all or almost all independent $a$-sets. This is because, amongst all possible choices of colour class sizes for a fixed number of colours $k \approx \frac{n}{2 \log_2 n}$, the expected number of colourings is maximised if all or almost all $a$-sets are included. This intuition indicates $\chi(G_{n, \frac 12})$ should vary at least 
as much as $X_a$ (up to a logarithmic factor). We show that this is indeed the case for \emph{some} values $n$ where $x=x(n)$ is close to $\frac 12$.

\section{Proof of Theorem \ref{theorem:nonconcentration}}
\subsection{Outline}
Suppose that $[s_n, t_n]$ is a sequence of intervals so that whp, $\chi(G_{n, \frac 12}) \in [s_n, t_n]$. In light of Theorem~\ref{theorem:bounds}, we may assume that
\begin{equation}s_n = f(n)+o\left(\frac{n}{\log^2 n}\right) \,\,\text{ where }\,\, f(n)=\frac{n}{2 \log_2 n-2 \log_2 \log_2 n - 2}.\label{eq:an}
\end{equation}
(It is clear that $s_n \le f(n) + o\left(\frac{n}{\log^2 n}\right)$, but if (\ref{eq:an}) does not hold, we may replace $s_n$ with some larger $s_n' = f(n)+o\left(\frac{n}{\log^2 n}\right)$, which can only shorten the interval lengths.) Let
\[
 l_n=t_n - s_n
\]
denote the \emph{interval length}, and fix $c \in \left(0, \frac 14\right)$. We will show that there is \emph{some} $n^* \ge 1$ such that
\[
 l_{n^*} \ge \left( n^*\right) ^c,
\]
which suffices to prove Theorem \ref{theorem:nonconcentration}.

We start in Section \ref{section:preliminaries} with the Poisson approximation of the distribution of $X_a$ and some technical lemmas. In Section \ref{section:typicalvalues}, we fix $\epsilon>0$ and pick some $n$ with $x=x(n) < \frac 12-\epsilon$, so that $\mu=\E[X_a]\le n^{\frac 12 -\epsilon}$. It will follow by a first moment argument that whp all independent $a$-sets in $G_{n, \frac 12}$ are \emph{disjoint}.

The proof now relies on comparing the chromatic numbers of $G_{n, \frac 12}$ and $G_{n', \frac 12}$, where $n'$ is slightly larger than $n$. Specifically, let $r=\left \lfloor \sqrt{\mu} \right \rfloor=\left \lfloor n^{x/2} \right \rfloor$ be roughly the standard deviation of $X_a$, let $n'=n+ a r$, and let $X_a'$ be the number of independent $a$-sets in $G_{n', \frac 12}$. In Section \ref{section:typicalvalues}, we will see that we can condition $G_{n, \frac 12}$ and $G_{n', \frac 12}$ on some typical values for $X_a$ and $X_a'$ which differ by exactly $r$, and on all independent $a$-sets being disjoint, so that the chromatic numbers of the conditional random graphs are still in the typical intervals $[s_n, t_n]$ and $[s_{n'}, t_{n'}]$ with significant probability.

In Section \ref{section:coupling}, we construct a coupling of essentially these two conditional random graph distributions so that the conditional $G_{n, \frac 12}$ is an induced subgraph of the conditional $G_{n', \frac 12}$ and their difference can be partitioned into exactly $r$ disjoint independent $a$-sets. Since the chromatic numbers of these two random graphs then differ by at most $r$ and both lie in the intervals $[s_n, t_n]$ and $[s_{n'}, t_{n'}]$ with positive probability, this implies $s_n'\le t_n+r$ or equivalently
\[ l_n\ge s_{n'}-s_n-r.\]
We then use the estimate $s_n=f(n)+o \left( \frac{n}{\log^ 2n} \right)$ given in (\ref{eq:an}). In (\ref{eq:difference}), we will see that
\[
 f(n')-f(n) \ge r+ \Theta\left( \frac r {\log n} \right).
\]
If the error term $o\left(\frac{n}{\log^2 n} \right)$ in (\ref{eq:an}) did not exist, this would immediately imply $l_n \ge \Theta\left( \frac r {\log n} \right)$. To beat the error term, we repeat the argument in Section \ref{section:finishing}  for a sequence $n_1<n_2<n_3<...$ of integers. Carefully checking that our assumptions remain valid throughout, we find some $n^* \geqslant n$ such that $l_{n^*}\geqslant \Theta \left(\frac{r^*}{\log n^*}\right) \geqslant \left(n^{*}\right)^{c}$.

\subsection{Preliminaries} \label{section:preliminaries}
Recall that $a=\left \lfloor \alpha_0 \right \rfloor$, where $\alpha_0=\alpha_0(n)$ is given in (\ref{eq:adef}), and that we denote by $X_a$ the number of independent $a$-sets in $G_{n, \frac 12}$, letting $\mu=\E[X_a]={n \choose a} \left(\frac 12\right)^ {a \choose 2}$. 
A standard calculation shows that for any function $h=h(n)=O(1)$ such that $\alpha_0-h$ is an integer, the expected number of independent sets of size $\alpha_0-h$ in $G_{n, \frac 12}$ is $n^{h+o(1)}$ (see \S3.c in \cite{mcdiarmid1989method}). Therefore, as $\alpha_0-1 < a \le \alpha_0$,
\begin{equation}\label{eq:xdef}
 \mu = n^x \,\,\text{ for some function }\,\, x=x(n) \in [o(1), 1+o(1)]
\end{equation}
which satisfies \begin{equation} \label{eq:xproperty}
x= \alpha_0-a+o(1)=\alpha_0-\left \lfloor \alpha_0 \right \rfloor +o(1).
                     \end{equation}

Since $\alpha_0(n)\rightarrow \infty$ and $(\alpha_0(n+1)-\alpha_0(n)) \rightarrow 0$ as $n \rightarrow \infty$, the following lemma is immediate.
\begin{lemma} \label{lemma:choiceofn}
Let $0 \le c_1 < c_2 \le 1$ and $N>0$. There is an integer $n \ge N$ such that
\[
 x(n) \in (c_1, c_2),
\]
where $x(n)$ is given by (\ref{eq:xdef}).\qed
\end{lemma}
The Stein-Chen method (see for example \S4 in \cite{ross2011fundamentals}) can be used to obtain some very accurate information about the distribution of $X_a$. For this, if $W,Z$ are two random variables taking values in a countable set $\Omega$, let
\[
\dtv(W,Z) = \sup_{A \subset \Omega} \left|\Pb(W \in A) - \Pb(Z \in A) \right|
\]
denote their \emph{total variation distance}. For $\lambda>0$, denote by $\Poi_\lambda$ the \emph{Poisson distribution} with mean~$\lambda$. The following lemma is a special case of Theorem 11.9 in \cite{bollobas:randomgraphs}.
\begin{lemma}\label{lemma:poissonapprox}Let $Z \sim \Poi_{\mu}$, then if $\mu\ge1$, 
\begin{equation*}
 \dtv(X_a,Z)  =O \Big(\mu(\log n)^4/n^2+(\log n)^3/n \Big)=o(1).
\end{equation*}\qed
\end{lemma}
We will also need a technical lemma about the Poisson distribution, a proof is given in the appendix.  
 For an integer $k$ and $\mc A \subset \Z$, let $\mc A-k = \{a-k \mid a \in \mc{A}\} $. 
\begin{lemma} \label{lemma:technicalPoisson} 
 Let $(\lambda_n)_n$ be a sequence with $0<\lambda_n \rightarrow \infty$, and suppose that $(\mc{B}_n)_n$ is a sequence of integer sets such that
 \[
  \Poi_{\lambda_n}\left(\mc{B}_n \right) \rightarrow 0.
 \]
Then also
\[
 \Poi_{\lambda_n}\left(\mc{B}_n-\left \lfloor\sqrt{\lambda_n} \right \rfloor\right) \rightarrow 0.
\]\qed
\end{lemma}

\subsection{Selection of typical values} \label{section:typicalvalues}
Now let $0<\epsilon<\frac 14$ be fixed, and suppose that $n$ is an integer such that
\begin{equation}\label{eq:xbound}
 \epsilon < x =x(n) < \frac 12 -{\epsilon},
\end{equation}
or equivalently $n^{\epsilon } < \mu < n^{\frac12 -\epsilon}$ (infinitely many such values $n$ exist by Lemma \ref{lemma:choiceofn}). Let 
\[r=  \left \lfloor n^{x/2} \right \rfloor \text{ and }  n'=n+ra.\]
For the rest of subsections \ref{section:typicalvalues} and \ref{section:coupling}, whenever we write $a$, $\alpha_0$, $\mu$ and $x$, this refers to $a(n)$, $\alpha_0(n)$, $\mu(n)$ and $x(n)$, respectively. Let  $a'=a(n')$, $\alpha_0'=\alpha_0(n')$ and $\mu'=\mu(n')$ (where $\mu'$ is defined with respect to $a'$).

Note that 
\[\alpha_0'=\alpha_0+O\left(\frac{r a}{n}\right)=\alpha_0+o(1).\]
In particular, as $\alpha_0-a=x+o(1)$ is bounded away from $0$ and $1$ by (\ref{eq:xproperty}) and (\ref{eq:xbound}), if $n$ is large enough we have
\[a'=\left \lfloor \alpha_0'\right \rfloor =\left \lfloor \alpha_0\right \rfloor=a.\]

As $\mu=n^x$, $r=O(n^{x/2})$, $a=O(\log n)$ and $x< \frac 12$, if $n$ is large enough,
\[
 \mu'={n' \choose a} \left( \frac 12\right)^{a \choose 2}=\mu \, \prod_{i=0}^{a-1} \frac{n'-i}{n-i} =\mu\left(1+O\left(\frac{ra}{n}\right)\right)^a=\mu \left(1+ O\left( \frac{r a^2}{n}  \right)\right) = \mu+o(1).
\]
In particular, an easy (and well-known) calculation shows that
\[
 \dtv(\Poi_\mu, \Poi_{\mu'}) =o(1).
\]
Recall that $X_a'$ denotes the number of independent $a$-sets in $G_{n', \frac 12}$, then together with Lemma \ref{lemma:poissonapprox}, it follows that $X_a$ and $X_a'$ have essentially the same distribution.
\begin{lemma} $\dtv(X_a, \Poi_{\mu})=o(1)$ and $\dtv(X_a', \Poi_{\mu})=o(1)$.\label{eq:poissonkonkret}
\qed
\end{lemma}
We would like to compare the chromatic numbers of $G_{n, \frac 12}$ and $G_{n', \frac 12}$, each conditioned on having ``typical'' numbers of independent $a$-sets which differ by exactly $r$, and conditioned on the event that all independent $a$-sets are disjoint (which holds whp). The content of the following lemma is that we can pick two such typical values for $X_a$ and $X_a'$ so that, after conditioning, the chromatic numbers of $G_{n, \frac 12}$ and $G_{n', \frac 12}$ still lie in their typical intervals with significant probability.
\begin{lemma} \label{lemma:valueA}
Let $G \sim G_{n, \frac 12}$ and $G' \sim G_{n', \frac 12}$. 
Let $\mc{E}$ and $\mc{E}'$ be the events that all independent $a$-sets in $G$  and $G'$ are disjoint, respectively. Then, if $n$ is large enough, there is an integer $A=A(n) \in [\frac{1}{2}n^x, 2n^{x}]$   
such that
\begin{align*}
 \Pb &\left( \chi(G) \in [s_{n}, t_{n}] \,\, \big| \,\, \{X_a = A\} \cap \mc{E} \right) > \frac 34 \,\,\text{ and }\\
\Pb &\left( \chi(G') \in [s_{n'}, t_{n'}] \,\, \big| \,\, \{X_a'=A+r\} \cap \mc{E}'\right) > \frac 34 .
\end{align*}
\end{lemma}
\begin{proof}
Let 
\begin{align*}
\mc{F} &=\{\chi(G) \in [s_{n}, t_{n}]\} \cap \mc{E} \\
\mc{F}' &=\{\chi(G') \in [s_{n'}, t_{n'}]\} \cap \mc{E}'.\end{align*}
Then (since $\Pb (\mc A \,|\, \mc B \cap \mc C) \ge \Pb (\mc A \cap \mc B \,|\, \mc C)$ for any events $\mc A$, $\mc B$, $\mc C$ and probability distribution $\Pb$), it suffices to show that there is a value $A$ such that 
\begin{align}
 \Pb \left(\mc{F} \,\,\big| \,\,X_a = A \right) &> \frac 34 \,\,\text{ and }  \label{eq:a1condition}\\
\Pb \left( \mc{F}' \,\,\big| \,\,X_a'=A+r\right) &> \frac 34 . \label{eq:a2condition}
\end{align}
As $\mu=n^x$ with $x< \frac 12 - \epsilon$ and $\mu'=\mu+o(1)$, an easy first moment calculation (for the number of pairs of independent $a$-sets which share between $1$ and $a-1$ vertices --- this is similar to the proof of Theorem 4.5.1 in \cite{alonspencer}) shows that the events $\mc{E}$ and $\mc{E}'$ both hold whp, and so the events $\mc{F}$ and $\mc{F}'$ also hold whp. 

Let $\mc{A}$ be the set of values $A$ for which (\ref{eq:a1condition}) holds, and let $\mc{A}'$ be the set of values $A$ for which (\ref{eq:a2condition}) holds. Then as
\begin{align*}
o(1) =  \Pb \left(\mc{F}^c \right) &\ge \sum_{A \notin \mc A} \Big( \Pb  \left(\mc{F}^c \,\,\big| \,\,X_a = A \right) \Pb(X_a = A) \Big) \ge \frac 14 \,\,\Pb (X_a \notin \mc{A}) \,\,\text{ and }\\
o(1) =  \Pb \left((\mc{F}')^c \right) &\ge \sum_{A \notin \mc A'} \Big( \Pb  \left((\mc{F}')^c \,\,\big| \,\,X'_a = A+r \right) \Pb(X'_a = A+r) \Big)  \ge \frac 14 \,\,\Pb (X_a' \notin \mc{A}'+r),
\end{align*}
whp $X_a \in  \mc{A}$ and $X_a' \in \mc{A}' +r$.

Therefore, by Lemma \ref{eq:poissonkonkret}, $\Poi_{\mu}(\mc{A})=1-o(1)$ and $\Poi_{\mu}(\mc{A}'+r)=1-o(1)$. From Lemma \ref{lemma:technicalPoisson} (applied to $\mc B_n = (\mc A '+r)^c$, noting that $\mc B_n - r = \mc A'^c $), it follows that also $\Poi_{\mu}(\mc{A}')=1-o(1)$, and so
\[
 \Poi_{\mu}\left(\mc{A} \cap \mc{A}'\right)=1-o(1).\]
Since $\mu> n^\epsilon \rightarrow \infty$,  by Chebyshev's inequality $\Poi_{\mu} \left( [\frac{1}{2}\mu, 2\mu] \cap \N_0\right)= 1-o(1)$. In particular, $\mc{A} \cap \mc{A}' \cap [\frac{1}{2}\mu, 2\mu]$ is non-empty, so there is at least one $A \in [\frac 12 \mu, 2\mu]$ which fulfils  (\ref{eq:a1condition}) and (\ref{eq:a2condition}).
\end{proof}

\subsection{Coupling the conditional distributions} \label{section:coupling}
Given an event $\mc P$, denote by $G_{n,p}|_{\mc P}$ the distribution of the random graph $G_{n,p}$ conditional on $\mc{P}$. The key ingredient of the proof is a construction of a coupling of essentially the two conditional distributions 
\[G_{n,\frac 12}|_{\{X_a=A\} \cap \mc{E}} \,\,\text{ and }\,\, G_{n', \frac 12}|_{\{X_a'=A+r\} \cap \mc{E}'}\]
from Lemma \ref{lemma:valueA}, so that the conditional $G_{n, \frac 12}$ is an induced subgraph of the conditional $G_{n', \frac 12}$ and their difference can be partitioned into $r$ independent $a$-sets.

For this, let $V'=[n']$, fix some arbitrary disjoint $a$-sets 
\[S_1, \dots, S_{A+r} \subset V',\]
as shown in Figure \ref{figure1}, and let 
\[V=V'\setminus \bigcup_{i=1}^r S_i.\]
Include every edge between vertices in $V'$ independently with probability $\frac 12$, and consider the following events.
\begin{align*}
 \mathcal D_1 :& \,\text{ The $a$-sets $S_{1}, \dots, S_{r}$ are independent.}\\
  \mathcal D _2 :& \,\text{ The $a$-sets $S_{r+1}, \dots, S_{r+A}$ are independent.}\\
   \mathcal U_1 :& \,\text{ There are no independent $a$-sets with at least one vertex in $V' \setminus V$, other than the $a$-sets}\\
  & \,\text{ $S_{1}, \dots, S_{r}$ (which may or may not be independent).}\\
  \mathcal U_2 :& \,\text{ There are no independent $a$-sets completely contained in $V$, other than the $a$-sets}\\
  & \,\text{ $S_{r+1}, \dots, S_{r+A}$ (which may or may not be independent).} 
\end{align*}
Note that $\mc U_1$ and $\mc U_2$ are up-sets, and $\mc D_1$ and $\mc D_2$ are principal down-sets (that is, $\mc D_1$ and $\mc D_2$ are events defined by forbidding a specific fixed edge set). Now condition on the event 
\[
 \mc D_1 \cap  \mc D_2 \cap  \mc U_1 \cap  \mc U_2
\]
that exactly the $A+r$ disjoint $a$-sets $S_1, \dots, S_{A+r}$ are independent and no others. Call the resulting random graph $H'$, and let $H=H'[V]$ be the induced subgraph of $H'$ on the vertex set $V$. By construction, $H$ is a graph on $n$ vertices with exactly $A$ disjoint independent $a$-sets, $H'$ is a graph on $n'$ vertices with exactly $A+r$ disjoint independent $a$-sets, and $V' \setminus V$ can be partitioned into $r$ disjoint independent $a$-sets.

\begin{figure}[tb]
\begin{center}
\begin{overpic}[width=0.8\textwidth]{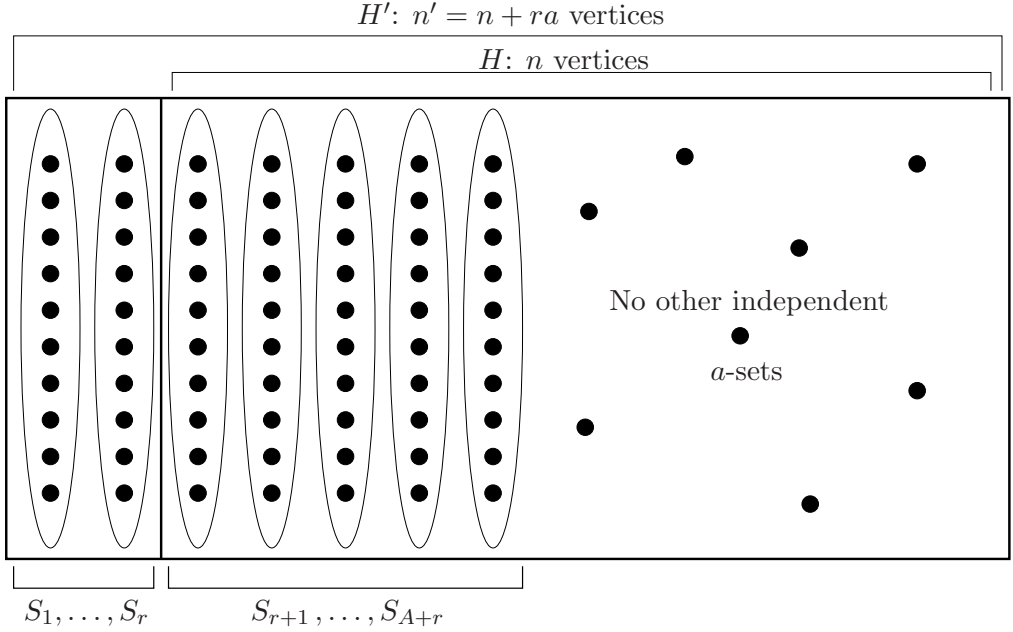}
\put(24.5,-3){$S_{r+1}\,, \dots, S_{A+r}$}
\put(2,-3){$S_1, \dots,  S_r$}
\put(35,56.5){{$H'$}: $n'=n+ra$ vertices}
\put(47,52){{$H$}: $n$ vertices}
\put(60,28){No other independent}
\put(70,21){$a$-sets}
\end{overpic}
\end{center}
\caption{Construction of the graphs $H$ and $H'$. We condition on the events that the fixed disjoint $a$-sets $S_1, \dots, S_{A+r}$ are independent, and that there are no other independent sets of size $a$.}
\label{figure1}
\end{figure}

It is not hard to see that, up to a random vertex permutation, $H'$ has exactly the required distribution.
\begin{claim} \label{claim1} Let $\hat H'$ be the random graph obtained from $H'$ by a uniform random permutation of the vertex labels in $V'$. Then
 $\hat H' \sim G_{n', \frac 12}|_{\{X_a'=A+r\} \cap \mc {E}'}.$ 
\end{claim}
\begin{proof}
Consider the random graph $G_{n', \frac 12}$ conditioned on $\{X_a'=A+r\}\cap \mc{E}'$. The set of all possible graphs on $n'$ vertices with exactly $A+r$ disjoint independent $a$-sets is the disjoint union of all such graphs where exactly $A+r$ \emph{fixed} disjoint independent $a$-sets are specified. In the conditional $G_{n', \frac 12}$, every such fixed collection of $A+r$ disjoint independent $a$-sets is equally likely (by symmetry).

Starting with the fixed collection $\mc S= \{S_1, \dots, S_{A+r}\}$ of $a$-sets, if $\pi$ is a uniform random permutation of $V'$, then the image $\pi(\mc S)$ is uniformly distributed amongst all collections of $A+r$ disjoint $a$-sets. Therefore, if we start by conditioning $G_{n', \frac 12}$ on having exactly the independent $a$-sets in $\mc{S}$ --- which is the distribution of $H'$ --- and then apply the random vertex permutation $\pi$, we recover the distribution $G_{n', \frac 12}|_{\{X_a'=A+r\} \cap \mc {E}'}$.
\end{proof}
Unfortunately, we cannot argue in the same way for $H$. If we obtain $\hat H$ from $H$ by randomly permuting the vertex labels in $V$, then $\hat H$ does \emph{not} have exactly the conditional distribution $G_{n, \frac 12} |_{\{X_a=A\} \cap \mc E}$. This is because the distribution of $H$ is also conditional on the event $\mc{U}_1$ that there are no other independent $a$-sets with at least one vertex in $V' \setminus V$.

However, as the expected number of such independent $a$-sets is small, the distributions are similar and we can bound probabilities in $H$ by the corresponding probabilities in $G_{n, \frac 12}|_{\{X_a=A\}}$. This can be deduced in several ways; the elegant formulation in Claim \ref{claim2} below was given by Oliver Riordan. 
 \begin{claim}\label{claim2}
 Let $\mc B$ be an event for the set of graphs with vertex set $V$ which is invariant under the permutation of vertex labels. Then 
 \[
  \Pb(H \in \mc B) \le (1+o(1)) \Pb \left( G_{n, \frac 12}|_{\{X_a=A\} \cap \mc {E}}\in \mc B \right) .
 \]
\end{claim}
\begin{proof}
By the same argument as in Claim \ref{claim1}, if we condition only on $\mc{D}_2 \cap \mc{U}_2$ and randomly permute the vertex labels of $V$, then the resulting random graph on $V$ has exactly the distribution $G_{n, \frac 12}|_{\{X_a=A\} \cap \mc {E}}$. Therefore, as $\mc B$ is invariant under the permutation of vertex labels,
\[
 \Pb \left( G_{n, \frac 12}|_{\{X_a=A\} \cap \mc {E}}\in \mc B \right) = \Pb(\mc{B} \mid\mc D_2 \cap  \mc U_2).
\]
The event $\mc{D}_1$ is independent from $\mc{B}$, $\mc D_2$ and $\mc U_2$ (as they depend on disjoint sets of edges), and so
\[
 \Pb \left( G_{n, \frac 12}|_{\{X_a=A\} \cap \mc {E}}\in \mc B \right) = \Pb(\mc{B} \mid \mc{D}_1 \cap \mc D_2 \cap  \mc U_2).
\]
Now note that
\begin{align*}
 \Pb( H \in \mc B) &= \Pb(\mc{B} \mid \mc D_1 \cap \mc D_2 \cap \mc U_1 \cap \mc U_2) = \frac{ \Pb(\mc{B} \cap \mc U_1 \mid \mc D_1 \cap \mc D_2 \cap \mc U_2)}{ \Pb(\mc U_1 \mid \mc D_1 \cap \mc D_2 \cap \mc U_2)} \le \frac{ \Pb(\mc{B} \mid \mc D_1 \cap \mc D_2 \cap \mc U_2)}{ \Pb(\mc U_1 \mid \mc D_1 \cap \mc D_2 \cap \mc U_2)}. \end{align*}
So to prove the claim, it suffices to show that $\Pb(\mc U_1 \mid \mc D_1 \cap \mc D_2 \cap \mc U_2) = 1-o(1)$. Note that $\mc D_1$ and $\mc D_2$ are principal down-sets, so after conditioning on $ \mc D_1 \cap \mc D_2 $, we still have a product probability space (for all the remaining edges which are not involved in  $\mc D_1 \cap \mc D_2$). As $\mc U_1$ and $\mc U_2$ are up-sets, by applying Harris' lemma to the aforementioned product space,
\[
 \Pb(\mc U_1 \mid \mc D_1 \cap \mc D_2 \cap \mc U_2) \ge \Pb(\mc U_1 \mid \mc D_1 \cap \mc D_2).
\]
So it suffices to show that $\Pb(\mc U_1^c \mid \mc D_1 \cap \mc D_2)=o(1)$. Note that $\mc U_1^c$ is the event that there is at least one independent $a$-set, other than  $S_1, \dots, S_r$, with at least one vertex in $V' \setminus V$. Let $Y$ denote the number of such sets. In Lemma \ref{lemma:Y} below, we will show by a straightforward but slightly involved calculation that
\begin{equation}\E[Y \mid \mc D_1 \cap \mc D_2]=o(1).\label{eq:Y}\end{equation}
This implies $\Pb(\mc U_1^c \mid \mc D_1 \cap \mc D_2)=o(1)$ as required.
\end{proof}
For the proof of Claim \ref{claim2} it remains to verify (\ref{eq:Y}).
\begin{lemma}\label{lemma:Y}
 Let $Y$ be as in the proof of Claim \ref{claim2}, then  $\E[Y \mid \mc D_1 \cap \mc D_2]=o(1)$.
\end{lemma}
\begin{proof}
We use the following notation below. For two functions $f, g: \N \rightarrow \R$, we write $f= O^*(g)$ if there are constants $C$ and $n_0$ such that $|f(n)| \le (\log n)^C g(n)$ for all~$n\ge n_0$. We write $f = \Theta^*(g)$ if $f = O^*(g)$ and $g= O^*(f)$.

Consider a potential $a$-set $T$ counted in $Y$. Then $T$ can be written as the disjoint union 
\[
 T= \bigcup_{j=1}^M T_j \cup T_\text{rest}
\]
where $M \ge 1$, $T_j \subset S_{i_j}$ for some $i_j \in \{1, \dots, A+r\}$, and $T_\text{rest} \subset V$. We can assume $i_1 < \dots < i_M$ and $i_1 \in [r]$ (as $T$ has at least one vertex in $V' \setminus V=\bigcup_{i=1}^r S_i$) and that the sets $T_j$ are non-empty for all $1\le j \le M$. As $T$ cannot be identical to any of the sets $S_1, \dots, S_{A+r}$, letting $t_j=|T_j|$, we have for all $1\le j \le M$,
\[
 1 \le t_j \le a-1.
\]
Let $\mc{T}$ be the set of all pairs $(M, \textbf{t})$, where $1\le M \le a$ is an integer and $\textbf{t}=(t_1, \dots, t_M)$ is a sequence of integers with $1\le t_j\le a-1$ for all $j$, and $\sum {t_j} \le a$. Starting with some $(M, \textbf{t}) \in \mc T$, an $a$-set $T$ corresponding to $(M, \textbf{t}) $ is defined by choosing $i_{1} \in [r]$, $\{i_2, \dots, i_M \} \subset [A+r]$, $T_{j} \subset S_{t_j}$ for all $1 \le j \le M$ and $T_{\text{rest}}\subset V$. So the number of $a$-sets $T$ corresponding to $(M, \textbf{t}) $ is at most
\begin{equation}r{A+r \choose M-1} \left(\prod_{j=1}^M {a \choose t_j} \right){n \choose a- \sum_{j=1}^M t_j} \le \frac{r}{A+r} {n \choose a} \prod_{j=1}^M  \frac{(A+r){a \choose t_j}a!}{(a-t_j)!(n-a)^{t_j}},
\label{eq:countT}
\end{equation}
using ${A+r \choose M-1} \le (A+r)^{M-1}$ and bounding
\[
 \frac{{n \choose a-\sum_j t_j}}  {{n \choose a}} = \frac{a!(n-a)!}{(a-\sum_j t_j)!(n-a+\sum_j t_j)!} \le \prod_{j=1}^M \frac{a!}{(a-t_j)!(n-a)^{t_j}}.
\]
Condition on $\mc D_1 \cap \mc D_2$, that is, on the event that none of the ${(A+r){a \choose 2}}$ edges within the $a$-sets $S_1, \dots, S_{A+r}$ are present. Then conditional on $\mc D_1 \cap \mc D_2$, a given $a$-set $T$ is independent with probability exactly $2^{-{a \choose 2}+\sum_j{t_j \choose 2}}$. Therefore, with (\ref{eq:countT}),
\begin{align}
 \E[Y \mid \mc D_1 \cap \mc D_2 ] &\le  \frac{r}{A+r} {n \choose a}2^{-{a \choose 2}} \sum_{(M, \textbf{t})\in \mc T} \prod_{j=1}^M \frac{(A+r){a \choose t_j}a! 2^{{t_j \choose 2}}}{(a-t_j)!(n-a)^{t_j}} =  \frac{r \mu}{A+r}  \sum_{(M, \textbf{t})\in \mc T} \prod_{j=1}^M \sigma_{t_j}, \label{eq:cont}
\end{align}
where
\[
 \sigma_t =   \frac{(A+r){a \choose t}a! 2^{{t \choose 2}}}{(a-t)!(n-a)^{t}} .
\]
Note that, as ${n \choose a-1}2^{-{a-1 \choose 2}}=\Theta\left(\frac{n}{\log n} \right) {n \choose a}2^{-{a \choose 2}}=\Theta\left(\frac{n}{\log n} \mu \right)$, $\mu=n^x$ and by Lemma \ref{lemma:valueA}, $A=O(n^x)$,
\begin{align*}
 \sigma_1 &= \Theta \left(\frac{n^x a^2}{n} \right)= \Theta^*\left(n^{x-1} \right),\\
 \sigma_{a-1}&=O \left(n^x a \right) \frac{a!2^{a-1 \choose 2}}{(n-a)^{a-1}} = O^*\left( \frac{n^x }{{n \choose a-1}2^{-{a-1 \choose 2}}}\right)=O^* \left(\frac{n^{x-1}}{\mu} \right) =O^*\left(n^{-1}\right)=o\left(n^{x-1} \right).
\end{align*}
Considering the ratio  $\sigma_{t+1}/\sigma_t =\frac{(a-t)^22^t}{(t+1)(n-a)}$ (which is $O^*\left(n^{-1} \right)$ for $t=O(1)$, then increases and reaches $\Theta^* \left( n\right)$ for $t=a-O(1)$), it is not hard to see that for all $1 \le t \le a-1$,
\[
 \sigma_t \le \max \left(\sigma_1, \sigma_{a-1} \right) =\sigma_1=O^*(n^{x-1}).
\]
So from (\ref{eq:cont}), as $1 \le t_j \le a-1$, $\mu=n^x$, $r=O(n^{x/2})$, $A=\Theta(n^x)$ and $x<\frac 12$,
\[
 \E[Y \mid \mc D_1 \cap \mc D_2 ] \le  \frac{r \mu}{A+r} \sum_{(M, \textbf{t})\in\mc T} \sigma_1^M \le \frac{r \mu}{A+r} \sum_{M \ge 1} \left(a^M \sigma_1^M \right)=O^* \left(\frac{r \mu \sigma_1}{A+r} \right) = O^* \left(n^{\frac{3}{2}x-1}\right)=o(1).
\]\end{proof}

As  the chromatic number of a graph is invariant under the permutation of vertex labels, it follows from Claim \ref{claim1} and Lemma \ref{lemma:valueA} that
\[
\Pb\left( \chi(H') \in [s_{n'}, t_{n'}] \right)>\frac 34.
\]
From Claim \ref{claim2} (applied to the event $\mc{B}=\{\chi(H) \notin [a_{n}, b_{n}]\}$) and Lemma \ref{lemma:valueA}, it follows that
\[
 \Pb\left( \chi(H) \in [s_{n}, t_{n}] \right) > (1+o(1)) \cdot \frac 34>\frac 12
\]
if $n$ is large enough. But as $V' \setminus V$ is the union of the independent $a$-sets $S_1, \dots, S_r$, we also have
\begin{equation*}
\chi(H') \le \chi(H)+r.
\end{equation*}
So with probability at least $\frac 14$,
\begin{equation*}
 s_{n'} \le \chi(H') \le \chi(H)+r \le t_n+r .
\end{equation*}
The left-hand side and the right-hand side are simply functions of $n$, not random variables, so it follows that, deterministically,
\[
s_{n'} \le  t_n+r
\]
and therefore
\begin{equation*}
 l_n=t_n-s_n \ge s_{n'}-s_n-r. 
\end{equation*}

\subsection{Finishing the proof} \label{section:finishing}

Let us summarise our progress so far in the following lemma. Recall the definitions of $s_n$ and $l_n$ given at the beginning of Section \ref{section:preliminaries}, and of the functions $a(n)$, $\mu(n)$ and $x(n)$ in Section \ref{section:preliminaries}.
\begin{lemma}\label{lemma:keylemma}
For every fixed $\epsilon \in (0,\frac{1}{4})$, if $n$ is large enough (i.e.\ if $n\ge N_\epsilon$ for some $N_\epsilon>0$) and
\[                                                                                                                                           
n^\epsilon \le \mu(n) \le n^{\frac 12 - \epsilon} \,\,\text{ or equivalently }\,\, \epsilon \le x(n) \le \frac 12 - \epsilon,                                                                                                                                             \]
then, 
letting $r=r(n)=\left \lfloor n^{x(n)/2} \right \rfloor $ and $n'=n+a(n)r$,
\begin{equation*} 
  l_n  \ge s_{n'}-s_n-r.
\end{equation*} \qed
\end{lemma}
Recall from  (\ref{eq:an}) that $s_n=f(n)+o\left( \frac{n}{\log^2 n}\right)$. With $n$ and $n'$ as in Lemma \ref{lemma:keylemma}, by a straightforward calculation which can be found in the appendix, if $n$ is large enough,
\begin{equation}
 f(n')-f(n) = r(n)+ \big(1-x(n)+o(1)\big)\frac{r(n)}{a(n)}  >  r(n)+ \frac{r(n)}{2a(n)}.\label{eq:difference}
\end{equation}
If we had equality in the estimate $s_n \approx f(n)$, without the additive error term $o\left( \frac{n}{\log^2 n}\right)$, then together with Lemma \ref{lemma:keylemma}, this would imply $l_n \ge \Theta \left( \frac{r(n)}{\log n} \right)$. To tackle the error term, we will apply Lemma~\ref{lemma:keylemma} to a \emph{sequence} of values $(n_i)_{i\ge 1}$.

To this end, let $c \in \left(0, \frac 14\right)$ be a constant, and let 
\begin{equation}\label{eq:defepsilon}
\epsilon=\frac{1}{4}\left(\frac 14-c\right) < \frac 1 {16}.\end{equation}
By Lemma \ref{lemma:choiceofn}, there is an arbitrarily large integer $n_1$ such that
\begin{equation}\label{eq:defn1}
2\epsilon < \frac12-4 \epsilon< x(n_1) < \frac 12 -3\epsilon. 
\end{equation}
Note that by (\ref{eq:xproperty}), $\alpha_0(n_1)=a(n_1)+x(n_1)+o(1)< a(n_1)+\frac 12-3 \epsilon+o(1)$. Let $M$ be the largest integer such that for all $n_1 \le n \le M$,
\begin{equation}
 \alpha_0(n) < a(n_1)+\frac 12 - 2\epsilon. \label{eq:al0}
\end{equation}
For $n_1$ large enough, it follows from the definition (\ref{eq:adef}) of $\alpha_0$ that
\begin{equation}M-n_1 =\Theta (n_1).\label{eq:Mn1}\end{equation}
Furthermore, if $n_1$ is large enough, then for all $n_1 \le n \le M$,
\begin{align}
\alpha_0(n_1) &\le \alpha_0(n) \le \alpha_0(M), \nonumber\\
 a(n)&= \left \lfloor \alpha_0(n)+o(1) \right \rfloor = a(n_1)\,\, \text{ by (\ref{eq:al0})}, \nonumber \\
 \epsilon < 2\epsilon+o(1)<x(n_1) \le x(n) &\le \frac 12 - 2\epsilon+o(1) < \frac 12 -\epsilon\,\, \text{ by (\ref{eq:xproperty}), (\ref{eq:defn1}) and (\ref{eq:al0}); and so}\nonumber\\
 n^\epsilon &< \mu(n) < n^{\frac 12 -\epsilon} \label{eq:properties}.
\end{align}
Let $a=a(n_1)$. We inductively define a sequence of integers: for $i \ge 1$, let $x_i=x(n_i)$, $r_i= \left \lfloor n^{x_i/2} \right \rfloor$ and
\begin{equation}\label{eq:ni+1}
 n_{i+1}= n_i + r_i a.
\end{equation}
Let $i_\mathrm{\max}$ be the largest index so that $n_{i_{\mathrm{max}}}\le M$. Note that if $n_1$ is large enough, by (\ref{eq:Mn1}),
\begin{equation}
 n_{i_\mathrm{max}}-n_1 = \Theta (n_1). \label{eq:diffbig}
\end{equation}
Now by the properties stated in (\ref{eq:properties}), if $n_1$ is large enough, we may apply Lemma \ref{lemma:keylemma} to every pair $(n_i, n_{i+1})$ where $1 \le i < i_\mathrm{\max}$. Let $s_i=s_{n_i}$, $t_i= t_{n_i}$ and $l_{i}=l_{n_i}$, then
\begin{equation*}
l_{i} \ge s_{i+1}-s_{i}-r_i \text{ for all }1\le i < i_\mathrm{\max}.
\end{equation*}
Therefore,
\begin{equation}\label{eq:first}
 \sum_{i=1}^{\imax-1}l_i \ge s_{\imax}-s_1 -\sum_{i=1}^{\imax-1} r_i.
\end{equation}
By (\ref{eq:difference}), if $n_1$ is large enough, for all $1 \le i < \imax$,
\begin{equation*}
 f(n_{i+1})-f(n_i) > r_i+  \frac{r_i}{2a},
\end{equation*}
and so by (\ref{eq:an}) and (\ref{eq:properties}), and as $n_{i_\mathrm{max}}=\Theta(n_1)$ by (\ref{eq:diffbig}),
\[
 s_\imax-s_1=f(n_{\imax})-f(n_1)+o\left(\frac{n_1}{\log^2 n_1} \right)> \sum_{i=1}^{\imax-1}\left(r_i+ \frac{r_i}{2a}  \right) +o\left(\frac{n_1}{\log^2 n_1} \right).
\]
Together with (\ref{eq:first}), this gives
\[
 \sum_{i=1}^{\imax-1} l_i >  \sum_{i=1}^{\imax-1} \frac{r_i}{2a} +o\left(\frac{n_1}{\log^2 n_1}\right).
\]
Note that by (\ref{eq:ni+1}) and (\ref{eq:diffbig}),
\[\sum_{i=1}^{\imax-1} \frac{r_i}{a} = \frac{n_\imax-n_1}{a^2}=\Theta \left(\frac{n_1}{a^2} \right)=\Theta \left(\frac{n_1}{\log^2 n} \right),\]
and so, if $n$ is large enough,
\begin{equation}
 \sum_{i=1}^{\imax-1} l_i >  \sum_{i=1}^{\imax-1} \frac{r_i}{2a} +o\left(\frac{n_1}{\log^2 n}\right)  \ge \sum_{i=1}^{\imax-1} \frac{r_i}{3a}.\label{eq:error}
\end{equation}
In particular, there is \emph{some} index $1\le i^* < \imax$ such that, letting $n^*=n_{i^*}$,
\[
 l_{n^*} > \frac{r_{i^*}}{3a}.
\]
(In fact, there are either many indices $i$ with $l_{n_i} > \frac{r_{i}}{10a}$, say, or some $i$ where $l_i$ is very long.)
By (\ref{eq:properties}), $r_{i^*}=\left \lfloor \left(n^*\right)^{x(n^*)/2} \right \rfloor \ge \left \lfloor \left(n^*\right)^{x(n_1)/2} \right \rfloor $. So by (\ref{eq:defepsilon}) and (\ref{eq:defn1}), if $n_1$ is large enough,
\[
 l_{n^*} > \left(n^* \right)^{\frac 14 - 4 \epsilon}=\left(n^*\right)^c,
\]
so we have found an integer $n^*$ with $l_{n^*}>\left(n^* \right)^c$ as required. \qed

\section{Remarks and open questions} \label{section:remarks}

\begin{itemize}
  \item As a corollary of Theorem \ref{theorem:nonconcentration}, the same conclusion holds for the random graph $G_{n,m}$ with $m= \left \lfloor n^2/4 \right \rfloor$, which was pointed out by Alex Scott. This is because we can couple the random graphs $G_{n,m}$ and $G_{n, \frac 12}$ so that whp their chromatic numbers differ by at most $\omega(n)\log n$ for any function $\omega(n) \rightarrow \infty$. 
  For this, start with $G_{n,m}$ and independently sample $E \sim \Bin \left({n \choose 2}, \frac 12 \right)$. Now either add $E-m$ edges to or remove $m-E$ from $G_{n,m}$ uniformly at random, so that the total number of edges is $E$. The new graph has the distribution $G_{n,\frac 12}$, and it is not hard to show that whp this changes the chromatic number by at most $\omega(n)\log n$. (Note that in both $G_{n,m}$ and $G_{n, \frac 12}$, an optimal colouring consists of $O \left( \frac n {\log n}\right)$ colour classes of size $O (\log n)$. If we add $|E-m|\le  n \sqrt{\omega(n)}$ random edges, then whp at most $\omega(n) \log n$ of these to ``spoil'' a given optimal colouring, which can be ``fixed'' by adding at most $\omega(n) \log n$ new colours.)  
\item Of course $X_a$ is not whp contained in any sequence of intervals of length less than $n^{\frac 12 - \epsilon}$ for any fixed $\epsilon>0$, because there are infinitely many values $n$ where $x(n) > 1-\epsilon$. We conjecture that the same is true for $\chi(G_{n, \frac 12})$. This exponent would match the upper bound for the concentration of $\chi(G_{n, p})$ given by Shamir and Spencer \cite{shamir1987sharp}.

In the proof of Theorem \ref{theorem:nonconcentration}, we only consider the case $x(n) <\frac 12 - \epsilon$ because then whp all independent $a$-sets in $G_{n, \frac 12}$ are disjoint. It is possible that the coupling argument could be refined to show that there is some interval $[s_n, t_n]$ of length at least $n^{\frac 12 - \epsilon}$. 
 \item While Theorem \ref{theorem:nonconcentration} was only proved for $p=\frac 12$, the same proof works for any constant $p \in (0, 1-1/e^2]$. For $p>1-1/e^2$, there are some additional technical difficulties because the estimate for $\chi(G_{n,p})$ given in \cite{heckel2018chromatic} differs from the one in Theorem \ref{theorem:bounds}, and we have not attempted this case.
 
It would be interesting to see whether the argument could be generalised to other ranges $p=p(n)$. As Alon and Krivelevich \cite{alon1997concentration} proved two point concentration for $p< n^{-\frac 12-\epsilon}$, this would be particularly interesting for $p$ close to $n^{-\frac 12}$.
 
\item It should be noted that the proof of Theorem \ref{theorem:nonconcentration} required a fairly good estimate for $\chi(G_{n, \frac 12})$. The error bound in Theorem \ref{theorem:bounds} is of size $o \left(\frac{n}{\log^2 n} \right)$, which is used in (\ref{eq:error}), and the proof would not have worked with an error bound of size $O\left(  \frac{n}{\log^2 n} \right)$. Therefore, to extend the result to other ranges of $p$, we might first need similarly accurate bounds for $\chi(G_{n,p})$.

 \item  Theorem \ref{theorem:nonconcentration} only implies that for any constant $c< \frac 14$, there are \emph{some} values $n$ where $l_n > n^{c}$. It could still be the case that $\chi(G_{n, \frac 12})$ is very narrowly concentrated on a subsequence of the integers. Can we find a lower bound for $l_n$ which holds for \emph{all} large enough~$n$?
 
 \item Ultimately, it would 
 be very nice to establish the correct exponent for the concentration of $\chi(G_{n, \frac 12})$, and it seems likely that this exponent varies with $n$. In other words, can we find a function $\rho(n)$ such that for any fixed $\epsilon>0$, $\chi(G_{n,p})$ is whp contained in some sequence of intervals of length $n^{\rho(n)+\epsilon}$, but for any sequence of intervals $I_n$ of length at most $n^{\rho(n)-\epsilon}$, if $n$ is large enough, 
  \[\Pb \left( \chi(G_{n, \frac 12}) \in I_n \right) < \frac 12?\]
\end{itemize}

\section*{Acknowledgements}
The work in this paper was completed during the Oberwolfach workshop ``Combinatorics, Probability and Computing'', and I am grateful to the MFO institute for their hospitality. I would like to thank Oliver Riordan for many discussions and suggestions which simplified 
the coupling argument considerably, as well as David Conlon, Eoin Long, Konstantinos Panagiotou and Lutz Warnke for several helpful discussions and remarks on earlier versions of this paper. I would also like to thank the anonymous referees for their detailed comments and suggestions which greatly improved the presentation of the paper.
 \bibliographystyle{plainnat}

\section*{Appendix}

\subsubsection*{Proof of Lemma \ref{lemma:technicalPoisson}}
To simplify notation, we write $\lambda$ instead of $\lambda_n$. Let $ r= \left \lfloor \sqrt{\lambda} \right \rfloor$ and fix $\epsilon>0$. Then it suffices to show that there is some $\delta=\delta(\epsilon)>0$ such that, if $n$ and therefore $\lambda$ is large enough, for all sets $\mc{B}$ with $\Poi_{\lambda}(\mc{B}) <\delta$ we have
\[
 \Poi_{\lambda}\left(\mc{B}- r \right) <\epsilon.
\]
So let $\delta>0$ be small enough so that 
\[
\log \left(\frac{\epsilon}{2\delta}\right)> \sqrt{\frac{3}{\epsilon}}+1,
\]
and pick some 
\begin{equation}
t \in \left( \sqrt{\frac{3}{\epsilon}}+1,  \log \left(\frac{\epsilon}{2\delta}\right)\right). \label{eq:deft}
\end{equation}
Suppose that $\mc{B}$ is a set of integers with $\Poi_{\lambda}(\mc{B}) <\delta$. For $s>0$, let 
\[
 I_s= \left[\lambda-sr, \lambda+s r \right]\cap \Z,
\]
and consider $\mc{B}_1=\mc{B} \setminus I_t$ and  $\mc{B}_2=\mc{B} \cap I_t$. Note that $ {I_{t-1}} \subset {I_t-r} $, so $\mc{B}_1 -r \subset I_t^c-r \subset I_{t-1}^c$. So by Chebyshev's inequality and (\ref{eq:deft}), if $\lambda$ is large enough,
\begin{equation}\label{eq:firstpart}
 \Poi_\lambda (\mc{B}_1-r) \le \Poi_\lambda (I_{t-1}^c) \le \frac{\lambda}{(t-1)^2r^2} < \frac{\epsilon}{2}.
\end{equation}
Now consider some $k \in \mc{B}_2$. If $k-r<0$, we have $\Poi_\lambda(\{k-r\})=0$, so in particular
\begin{equation}\label{eq:asjk}
 \Poi_\lambda (\{k-r\}) \le \frac{\epsilon}{2\delta} \Poi_\lambda(\{k\}).
\end{equation}
Otherwise, as $k \le \lambda+t r \le \lambda + t\sqrt{\lambda}$, by (\ref{eq:deft}), 
\[
\frac{\Poi_\lambda (\{k-r\})} {\Poi_\lambda (\{k\}) } = \frac{ k!}{ \lambda^r(k-r)!} \le \left( \frac{k}{\lambda}\right)^r \le \left( 1+\frac{t}{\sqrt{\lambda}} \right)^{\sqrt{\lambda}} \le e^{t}< \frac{\epsilon}{2\delta}.\]
 Therefore, together with (\ref{eq:asjk}),
 \[
  \Poi_\lambda (\mc{B}_2-r) \le \frac{\epsilon}{2\delta} \Poi_\lambda (\mc{B}_2)\le \frac{\epsilon}{2\delta} \Poi_\lambda (\mc{B}) < \frac{\epsilon}{2} .
 \]
 Together with (\ref{eq:firstpart}), this proves the claim.
\qed

\subsubsection*{Proof of (\ref{eq:difference})}
First note that 
\[f(n)=\frac{n}{2 \log_2 n - 2 \log_2 \log_2 n-2}=\frac{n}{\alpha_0(n)-1-\frac{2}{\log 2}}
.\]
Now
\begin{equation}\label{eq:alphan}
 \alpha_0(n')-\alpha_0(n) \sim 2 \log_2 n' - 2 \log_2 n = \frac{2}{\log 2}\log \left(1+\frac{ar}{n}\right) \sim \frac{2ar}{n \log 2}=o(1).
\end{equation}
Note that
\begin{align}
 f(n')-f(n)&= \frac{n'-n}{\alpha_0(n')-1-\frac{2}{\log 2}} +\left(\frac{n}{\alpha_0(n')-1-\frac{2}{\log 2}}-\frac{n}{\alpha_0(n)-1-\frac{2}{\log 2}} \right).\label{eq:twoparts}
\end{align}
For the first term in (\ref{eq:twoparts}), since $\alpha(n')=\alpha(n)+o(1)=a+x+o(1)$ by (\ref{eq:xproperty}),
\begin{equation}
 \frac{n'-n}{\alpha_0(n')-1-\frac{2}{\log 2}}=\frac{ar}{a+x-1 -\frac{2}{\log 2}+o(1)}= r+ \frac{\left(1+\frac{2}{\log 2}-x \right)r}{a} +o \left(\frac{r}{a} \right). \label{eq:jds}
\end{equation}
For the second term in (\ref{eq:twoparts}), note that by (\ref{eq:alphan}),
\begin{align*}
 \frac{n}{\alpha_0(n')-1-\frac{2}{\log 2}}-\frac{n}{\alpha_0(n)-1-\frac{2}{\log 2}}&= \frac{n}{\alpha_0(n)-1-\frac{2}{\log 2}+\frac{2ar}{n \log 2}+o \left(\frac{ar}{n} \right)}-\frac{n}{\alpha_0(n)-1-\frac{2}{\log 2}}\\
 &\sim \frac{-\frac{2ar}{\log 2}}{\alpha_0(n)^2}\sim-\frac{2r}{a\log 2}
\end{align*}
as $\alpha_0(n) \sim a$. Together with (\ref{eq:twoparts}) and (\ref{eq:jds}), this gives
\[
 f(n')-f(n)= r+\frac{\left(1-x \right)r}{a} +o \left(\frac{r}{a} \right). 
\]\qed
\end{document}